\documentclass[12pt,a4paper]{amsart}

\usepackage{enumerate}

\usepackage{mathpazo}

\newcommand{\dd}{\mathrm{d}}
\newcommand{\real }{{\mathbb R}}

\newcommand{\lorentz }{{\mathbb L}}
\newcommand{\mink}{{\real_{1}^{4}}}
\newcommand {\lpr}[2]{{\langle}{#1},{#2}{\rangle}} 
\newcommand {\D}{{\mathcal D}}

\theoremstyle{plain}
\newtheorem{theo}{Theorem}[section]
\newtheorem{prop}[theo]{Proposition}
\newtheorem{lemma}[theo]{Lemma}
\newtheorem{corol}[theo]{Corollary}

\theoremstyle{definition}
\newtheorem{defn}[theo]{Definition}
\newtheorem{question}[theo]{Problem}
 \newtheorem{example}[theo]{Example}

\DeclareMathOperator{\Span}{span}
\DeclareMathOperator{\sech}{sech}
\DeclareMathOperator{\cosec}{cosec}

\begin{document}

\title{Chebyshev Nets in $\real^3$ and Minimal Timelike Surfaces in $\mink$}

\author{Antonio de Padua Franco Filho\and Alexandre Lymberopoulos}
\address{Departamento de Matem\'atica - Universidade de S\~ao Paulo}
\email{apadua@ime.usp.br, lymber@ime.usp.br}

\begin{abstract}
  In this work we provide necessary and sufficient conditions for the
  existence of a minimal timelike strip in Lorentz-Minkowski space
  $\mink$ containing a given lightlike curve and prescribed normal
  bundle. We also discuss uniqueness of solutions.
\end{abstract}

\maketitle

\section{Introduction}

The classical Bj\"orling problem can be formulated as follows: given a
real analytic curve $\alpha\colon I\subset\real\to\real^3$ and a unit
normal vector field $V\colon I\to\real^3$, along $\alpha$, determine a
minimal surface containing $\alpha(I)$ such that its normal vector along
the curve is $V$. The problem was firstly proposed and solved by
Bj\"orling himself in \cite{BJO} (1844). It was mentioned by Schwarz in
\cite{SCH} (1875) who also solved it, using a representation based on
holomorphic data, in \cite{SCH2} (1890).

Since then, many generalizations of this problem appeared in several
Riemannian and pseudo-Riemannian ambient manifolds. In $\real^3_1$
Alías, Chaves and Mira studied maximal spacelike surfaces in ~\cite{ACM}
and timelike minimal surfaces were studied by Chaves, Dussan and Magid
in ~\cite{CDM}, where both existence and uniqueness of solutions are
established. Analogous results are proved in $\mink$, for spacelike
surfaces in \cite{AV} by Asperti and Vilhena and, for timelike surfaces,
in \cite{DFM} by Dussan, Padua and Magid. The same holds for timelike
surfaces in $\real^4_2$ (see \cite{DM}). On Riemannian or Lorentzian Lie
Groups, Mercuri and Onnis, in \cite{MO}, and Cintra, Mercuri and Onnis,
in \cite{CMO}, also obtained results on existence and uniqueness of
solutions. In all those papers the authors make use of some kind of
Weierstrass representation formula, over complex or split-complex
domains.

The study of timelike minimal surfaces is important not only from the
mathematical point of view but also in physics, since they are solutions
for the wave equation and therefore can be regarded (classical)
relativistic strings.

In this work, without use of those complex or split-complex
representations, we provide necessary and sufficient conditions for the
existence of a solution for the Bj\"orling problem for a timelike
surface in $\mink$, when the prescribed curve is lightlike. In this case
we cannot expect uniqueness of solutions, which will be shown to be a
certain lift of a Chebyshev net in euclidean space $\real^3$.

\section{Algebraic preliminaries and the two kinds of Chebyshev nets}

The space $\real_{1}^{4}$ is the vector space $\real^{4}$ equipped with
the following semi-Riemannian metric tensor:
\[\dd s_{1}^{2} = - \dd x^{0} \otimes \dd x^{0} + \dd x^{1} \otimes \dd
  x^{1} + \dd x^{2} \otimes \dd x^{2} + \dd x^{3} \otimes \dd x^{3}.\]
We write this tensor in the inner product notation
$\langle v,w\rangle=\dd s_1^2(v,w)$. The standard basis of $\mink$ will
be denoted by $\{\partial_{0},\partial_{1},\partial_{2}, \partial_{3}\}$
and we set $\epsilon_i=\langle \partial_i,\partial_i\rangle$. If
$v=\sum\limits_{i=1}^4v_i\partial_i$, we have
$v_i=\epsilon_i\langle v,\partial_i\rangle$. A vector $v\in\mink$ is
\emph{spacelike} if $v=0$ or $\langle v,v\rangle>0$, \emph{timelike} if
$\langle v,v\rangle<0$ and \emph{lightlike} if $v\neq 0$ and
$\langle v,v\rangle=0$. In the same way, a \emph{spacelike plane} $V$ of
the space $\mink$ is a $2$-dimensional subspace for which the induced
bilinear form, $\dd s_{1}^{2}\big|_V$, is positive definite; we say that
$V$ is \emph{timelike plane} if $\dd s_{1}^{2}\big|_V$ is non-degenerate
and indefinite and it is \emph{lightlike} if $\dd s_{1}^{2}\big|_V$ is
degenerate.

Let $\{a,b\}$ be an orthonormal basis of a spacelike plane
$V \subset \mink$ and consider the unit timelike vector
\begin{equation}\label{eq:tau}
  \tau = \frac{1}{\sqrt{1 + a_{0}^{2} + b_{0}^{2}}} \left(\partial_{0} +
    a_{0}a + b_{0}b\frac{}{}\right)
\end{equation}
The standard wedge product of $u,w,w\in\mink$ is
$u\wedge v\wedge w \in\mink$, the unique solution for
$\langle u\wedge v\wedge w,x\rangle=\det(x,u,v,w)$. In matrix notation
we have the formal determinant \[u\wedge v\wedge w=\det
  \begin{pmatrix}
    \partial_{0}&\partial_{1}&\partial_{2}&\partial_{3}\\
    u_0&u_1&u_2&u_3\\
    v_0&v_1&v_2&v_3\\
    w_0&w_1&w_2&w_3
  \end{pmatrix}.\] Setting $\Delta_{ij} = a_{i}b_{j} - a_{j}b_{i}$ for
$0\leq i,j \leq 3$, we have the unit spacelike vector
\begin{equation}\label{eq:nu}
  \nu = -\tau \wedge a \wedge b = \Delta_{23} \partial_{1} - \Delta_{13}
  \partial_{2} + \Delta_{12} \partial_{3}.
\end{equation}

The $2$-dimensional vector subspace $T=\Span\{\tau,\nu\}$ is a timelike
plane which is the orthogonal complement of $V$. The $4$-uple
$(\tau,a,b,\nu)$ is a positive and future-directed frame, name
\emph{Minkowski frame} adapted to $\{a,b\}$.

Indeed, we see that $\lpr{\tau}{\tau} = -1$ and
$\tau_{0} = \sqrt{1 + a_{0}^{2} + b_{0}^{2}}\geq 1$, with
$\lpr{\tau}{a} = 0 = \lpr{\tau}{b}$. We also have that $\nu_{0} = 0$,
and $\lpr{\nu}{\nu} = 1$, because the set $\{\tau,a,b\}$ is an
orthonormal subset of $\mink$.  For each lightlike vector
$L = (L_{0},L_{1},L_{2},L_{3})$ we define its projection onto the unit
sphere $S^{2} \subset \{0\} \times \real^{3}$ by the formula
\begin{equation}
  \label{eq:projection}
  \pi(L) = (0, L_{1}/L_{0},L_{2}/L_{0},L_{3}/L_{0}).
\end{equation}

The vectors $\tau \mp \nu$ are lightlike. Hence we set
\begin{align}
  \label{eq:n0n3}
  n_{0} &= \pi(\tau - \nu) = (1/\tau_{0})(\tau - \nu) -
  \partial_{0}\mbox{ and}\nonumber\\
  n_{3} &= \pi(\tau + \nu) =
  (1/\tau_{0})(\tau + \nu) - \partial_{0}
\end{align}
to define a trigonometric angle $\theta \in ]0, \pi]$ in $V$ by
\begin{equation}
  \cos \theta = \lpr{n_0}{n_3} = 1 -
  \frac{2}{\tau^{2}_{0}}=\frac{a_{0}^{2} + b_{0}^{2}-1}{a_{0}^{2} +
    b_{0}^{2}+1}.
\end{equation}

\begin{prop}\label{prop:angle_theta}
  For the angle $\theta$ above we have \[\sin \theta =
    \frac{2\sqrt{a_{0}^{2} + b_{0}^{2}}}{\tau_0^2},\quad
    \sin(\theta/2)=\frac{1}{\tau_0},\quad\mbox{and}\quad\cos(\theta/2)=\frac{\sqrt{a_0^2+b_0^2}}{\tau_0}.\]
\end{prop}

The timelike plane
$T = \Span\{\partial_{0} + n_{0},\partial_{0} + n_{3}\}$ has induced
metric tensor represented, in this isotropic basis, by
\[g_{ij}=
  \left[\begin{matrix}
      0 & -1 + \cos \theta \\
      -1 + \cos \theta & 0
    \end{matrix} \right]=
  \left[\begin{matrix}
      0 & -2/\tau_{0}^{2} \\
      -2/\tau_{0}^{2} & 0
    \end{matrix} \right].\]

In the spacelike plane
$E = \Span\{n_{0},n_{3}\} \subset \{0\} \times \real^{3}$, with respect
to the given basis, it has the form
\[\hat{g}_{ij} =
  \left[\begin{matrix}
      1 & \cos \theta \\
      \cos \theta & 1
    \end{matrix} \right] =
  \left[\begin{matrix}
      1 & 1-2/\tau_{0}^{2} \\
      1-2/\tau_{0}^{2} & 1
    \end{matrix} \right].\]

Now, when $\tau_{0} > 1$ (that is,
$\vert a_{0} \vert + \vert b_{0} \vert \neq 0$) we define an orthonormal
basis $\{\tilde{e}_{1},\tilde{e}_{2}\}$ for the plane $V$ by
\begin{equation}\label{eq:convenient_base}
  \tilde{e}_{1} = \frac{1}{\sqrt{a_{0}^{2} + b_{0}^{2}}}(a_{0}a +
  b_{0}b)\quad\mbox{and}\quad
  \tilde{e}_{2} = \frac{1}{\sqrt{a_{0}^{2} + b_{0}^{2}}}(-b_{0}a + a_{0}b).
\end{equation}  

We note that $\Span\{\tilde{e}_{2}\} = V \cap \{0\} \times
\real^{3}$. Setting
\begin{equation}
  \label{eq:e-n0n3}
  e = \frac{1}{2 \cos(\theta/2)}(n_{0} + n_{3}) \in S^{2}
\end{equation}
we have the following result.

\begin{prop}
  On the above conditions, the following relations on the vectors of the
  (non-orthogonal) Minkowski frame $\{\tau, \tilde{e}_{1},e, \nu\}$
  hold:
  \begin{align*}
    \tau &= \frac{1}{\tau_0}\big(\partial_{0} +\sqrt{a_{0}^{2} +
    b_{0}^{2}}\;\tilde{e}_{1}\big) =
    \tau_0\partial_{0}+\tau_0\cos(\theta/2)\; e\mbox{ and}\\
    \tilde{e}_{1} &= \cot
    (\theta/2) \; \partial_{0} + \cosec (\theta/2) \; e.
  \end{align*}
\end{prop}
\begin{proof}
  The first identity comes from equations (\ref{eq:n0n3}) and
  (\ref{eq:e-n0n3}), where we see that
  \[\cos(\theta/2)\;e=\frac{n_0+n_3}{2}=\tau/\tau_0-\partial_0.\]

  For the second one, observe that $\tilde e_1$ is orthogonal to $\tau$
  and $\nu$. This means that $\tilde e_1=\alpha\partial_0+\beta e$, for
  some $\alpha, \beta\in\real$. From Proposition~\ref{prop:angle_theta},
  since $\partial_0$ and $e$ are mutually orthonormal, we have
  \begin{align*}
    \alpha&=-\lpr{\tilde{e}_1}{\partial_0}=\sqrt{a_0^2+b_0^2}=\cot(\theta/2)\\
    \beta&=\lpr{\tilde{e}_1}{e}=\tau_0=\cosec(\theta/2),
  \end{align*}
  as stated.
\end{proof}

Now, we will define Chebyshev nets as immersions in the Euclidean
vector space $\mathbb{E} = \{0\} \times \real^{3} \subset \mink$.

\begin{defn}
  We say that an immersion $(M,X)$ from a connected open subset
  $M \subset \real^{2}$ into the Euclidean space $\mathbb{E}$ is a
  Chebyshev net if and only if the coefficients of its first quadratic
  form, written as
  $\dd s^{2} = E(u,v)\,\dd u^{2} + 2F(u,v)\,\dd u \dd v + G(u,v)\,\dd
  v^{2}$, verifies, for all $(u,v) \in M$,
  \[E(u,v) = G(u,v) = 1 \mbox{ and } F(u,v) = \cos \theta(u,v) \in
    ]-1,1[.\]

  Associated to each Chebyshev net $(M,X)$ there is a timelike
  isotropic immersion $(M,f)$, \emph{the lift of $X$}, from $M$ into $\mink$
  defined by the formula
  \[f(u,v) = (u + v) \partial_{0} + X(u,v),\] whose induced metric
  tensor is
  \[g_{ij}(f) = \left[
      \begin{matrix}
        0 & -1 + F \\ -1 + F & 0 
      \end{matrix} \right] = \left[ 
      \begin{matrix}
        0 & -2\sin^{2}(\theta/2) \\ -2 \sin^{2}(\theta/2) & 0 
      \end{matrix} \right]\]
\end{defn}

If $(M,X)$ is a Chebyshev net, we consider the equivalent immersion
$(\overline{M},\overline{X})$ obtained applying the linear change of
coordinates $T\colon  \real^{2} \to\real^{2}$ given by:
\[t = u + v\quad\mbox{and}\quad s = -u + v,\mbox{ such that } \dd t \wedge
  \dd s = 2\, \dd u \wedge \dd v.\] That is, $\overline{M} = T(M)$ and
\begin{equation}
  \label{eq:equivimmersion}
  \overline{X}(t,s) = X\left(\frac{t - s}{2},\frac{t + s}{2}\right)=X(u,v).
\end{equation}

Now the metric tensor is given by
\[\dd s^{2}_{\overline{X}} = \overline{E}\,\dd t^{2} + \overline{G}\,\dd s^{2} =
  \cos^{2} (\theta/2)\,\dd t^{2} + \sin^{2} (\theta/2) \,\dd s^{2}.\]

The correspondent lift immersion
\[\overline{f}(t,s) = t \partial_{0} + \overline{X}(t,s)\] 
has isothermal parameters and its induced metric is
\[\dd s^{2}_{\overline{f}} = \sin^{2}(\theta/2)(-\dd t^{2} + \dd
  s^{2}).\]

\begin{theo}\label{teo:tcheblift_geom}
  Let $f(u,v) = (u + v) \partial_{0} + X(u,v)\in\mink$ be a lift of a
  Chebyshev net. The vector fields
  \begin{equation}\label{eq:normaltcheb1}
    \tilde{e}(u,v) = \frac{1}{\sin \theta(u,v)}\big((1 + \cos \theta(u,v)\big)
    \partial_{0} + X_{u}(u,v) + X_{v}(u,v)\big)
  \end{equation}
  and 
  \begin{equation}\label{eq:normaltcheb2}
    e_{2}(u,v) = \frac{1}{\sin \theta(u,v)} X_{u}(u,v) \times_{\real^{3}} X_{v}(u,v)
  \end{equation}
  form a spacelike orthonormal normal frame along $S = f(M)$.
  Moreover,
  the mean curvature vector $H_{f}(u,v)$ of the surface
  $S$ is pointwise parallel to the normal Gauss map $e_{2}(u,v)$ of
  the surface $X(M)\subset\mathbb{E}$.
\end{theo}

\begin{proof}
  Straightforward computations, using Chebyshev net properties, show
  the algebraic aspects of the statement.

  The coefficients of induced metric tensor on $f(M)$ give the mean
  curvature vector
  \begin{equation}
    \label{eq:Hlift}
    H_f=\frac{-1}{2\sin^2(\theta/2)} f_{uv} =
    \frac{-1}{2\sin^2(\theta/2)} X_{uv},
  \end{equation}
  which is orthogonal to $\tilde e$, hence parallel to $e_2$.
\end{proof}

\begin{prop}\label{prop:K_lift}
  The Gaussian curvature of a lift such as in Theorem
  \ref{teo:tcheblift_geom} is
  \begin{equation}
    \label{eq:K_lift}
    K=\frac{\theta_u\theta_v-\theta_{uv}\sin\theta}{(1-\cos\theta)^2}.
  \end{equation}
\end{prop}

\begin{proof}
  From~\cite[p. 443]{TL}, the Gaussian curvature of a parametric surface
  whose coordinates curves are lightlike is given
  by \[K=-\frac{1}{g_{12}}\left(\frac{(g_{12})_u}{g_{12}}\right)_v.\] In
  this case $g_{12}=-1+\cos\theta$.
\end{proof}

Recall that Gaussian curvature of any Chebyshev net satisfies the
equation $\theta_{uv}+K_T\sin\theta=0$. Hence we may rewrite
(\ref{eq:K_lift}) as
\begin{equation}
  \label{eq:K_liftK_T}
  K=\dfrac{\theta_u\theta_v+K_T\sin^2\theta}{(1-\cos\theta)^2}
\end{equation}

Now we will give two examples of Chebyshev nets, the first has a lift
with $H_{f} \equiv 0$ and the second is not a critical surface of
$\mink$.

\begin{example}[Critical lift]\label{ex:crit_tcheby}
  Set $U = ]-\pi/2,\pi/2[^{2}$ and consider the immersion
  $X:U\to\mathbb{E}$, given by 
  \[X(u,v) = \int_{0}^{u} (0,\cos \xi,\sin \xi,0)\,\dd \xi + \int_{0}^{v}
    (0,0,\sin \xi,\cos \xi)\, \dd \xi.\] Direct calculations show that:
  \begin{enumerate}
    \item the first quadratic form or metric tensor is
    \[\dd s^{2} = \dd u^{2} + 2\sin u \sin v \, \dd u \dd v + \dd v^{2};\]
    \item the normal Gauss map is
    \[e_{2}(u,v) = \frac{1}{\sqrt{1 - \sin^{2} u \sin^{2} v}}(0,\sin u
      \cos v, - \cos u \cos v, \cos u \sin v);\]
    \item the second quadratic form is
    \[B = \frac{-1}{\sqrt{1 - \sin^{2} u \sin^{2} v}}(\cos v du^{2} +
      \cos u dv^{2});\mbox{ and}\]
    \item the Gaussian curvature is
    \[K(u,v) = \frac{\cos u \cos v}{(1 - \sin^{2} u \sin^{2} v)^{2}} >
      0.\]
  \end{enumerate}

  The lift surface, $f(u,v) = (u + v) \partial_{0} + X(u,v)$, has
  vanishing mean curvature: one can see this from $X_{uv}=0$ in
  (\ref{eq:Hlift}) or noting that $f$ is a sum of two lightlike curves
  (see~\cite[p. 68]{C}).
\end{example}

\begin{lemma}
  Let $(W,Y)$ be an immersion from a connected open subset
  $W \subset \real^{2}$ into $\mathbb{E}$ with induced metric given by
  \[\dd s^{2}_Y = E(t,s)\,\dd t^{2} + G(t,s)\,\dd s^{2}.\] The equivalent
  immersion $(M,X)$ defined by $X(u,v) = Y(u + v,-u + v)$ is a
  Chebyshev net if and only if \[E(t,s) + G(t,s) = 1.\]
\end{lemma}

\begin{proof}
  We only need to observe that:
  \begin{align*}
    X_{u}(u,v) &= Y_{t}(u + v,-u + v) - Y_{s}(u + v,-u + v),\\
    X_{v}(u,v) &= Y_{t}(u + v,-u + v) + Y_{s}(u + v,-u + v).
  \end{align*}
  
  Hence
  \begin{align*}
    \overline{E}(u,v) &= \overline{G}(u,v) = E(t,s) + G(t,s)\mbox{ and}\\
    \overline{F}(u,v) &= E(u + v, -u + v) - G(u + v, -u + v).
  \end{align*}

  If $E(t,s) + G(t,s)=1$ then $\overline{E}(u,v)= \overline{G}(u,v)=1$
  and, since $\vert \overline{F}(u,v) \vert \leq 1$, we have a smooth
  real valued function $\theta(u,v)$ from $M$ such that
  $F(u,v) = \cos \theta(u,v)$. The converse is trivial.
\end{proof}

\begin{example}[Non-critical lift]\label{ex:non_crit_tcheby}
  Let $Y\colon ]-\pi,\pi[ \times I\to\mathbb{E}$ be the parametric surface
  given by
  \[Y(t,s) = \big(0, x(s) \cos t, x(s) \sin t, y(s)\big).\] We have that
  the metric coefficient $F$ verifies $F(t,s) = 0$. Suppose that the
  other coefficients satisfy $E(t,s) + G(t,s) = 1$ and . In this case,
  the lift surface $f(t,s) = t \partial_{0} + Y(t,s)$ is isothermal and
  timelike.  In terms of equation (\ref{eq:Hlift}), to obtain a non
  critical surface we must have the equivalent immersion $X(u,v)$
  satisfying $X_{uv}\neq 0$, that is,
  $f_{tt}-f_{ss}=Y_{tt}-Y_{ss}\neq 0$. The ordinary differential
  equation imposed by the condition $E(t,s) + G(t,s) = 1$ is
  \[x^{2}(s) + (x'(s))^{2} + (y'(s))^{2} = 1.\]
  
  The functions
  \[x(s)=\frac{1}{2}\tanh s\quad\mbox{and}\quad y(s) = \frac{1}{2} \int_0^s
    \sqrt{4 - \tanh^{2} \xi - \sech^{4} \xi}\,\dd\xi,\] are a
  particular solution to this equation. Since, $y'' \neq 0$, we have
  $f_{tt}-f_{ss}\neq 0$ and $H_f\neq 0$.% Now, applying the
\end{example}

\begin{defn}
  We say that a Chebyshev net $(M,X)$ is a Chebyshev net of first kind
  if and only if
  \[X(u,v) = p_{0} + \int_{0}^{u} T_{1}(\xi)\, \dd \xi + \int_{0}^{v}
    T_{2}(\xi)\, \dd\xi,\] for any disjoint curves
  $T_{1} \colon  I \to S^{2}\subset\mathbb{E}$ and
  $T_{2} \colon  J \to S^{2}\subset\mathbb{E}$ such that
  \[\{(u,v) \in I \times J \colon  T_{1}(u) = T_{2}(v)\} \cup \{(u,v) \in I
    \times J \colon  T_{1}(u) = -T_{2}(v)\} = \emptyset.\]
\end{defn}

\noindent\emph{Remark}: Example \ref{ex:crit_tcheby} above uses a Chebyshev net of
first kind.

\section{The Cauchy problem for Chebyshev nets and timelike minimal
  surfaces in $\mink$}

\begin{question}\label{problem:bjorling}
  Given a real analytic lightlike curve
  $c\colon ]-r,r[ \subset \real\to\mink$ and a spacelike distribution
  $\D(t) = \Span\big\{m(t),n(t)\big\}$ normal along this curve, establish necessary and
  sufficient conditions for the existence of a timelike minimal
  immersion $(M,f)$ from an open and connected subset $M$, where
  $I \times \{0\} \subset M \subset \real^{2}$, such that
  \begin{enumerate}
    \item the curve $c$ is the coordinate curve $f(t,0) = c(t)$,
    \item the normal bundle of $f(M)$ is the given distribution:
    $N_{c(t)}f(M) = \D(t)$.
  \end{enumerate}

  What can we say about uniqueness?
\end{question}

We start obtaining an integral representation for an isotropic timelike
minimal parametric surface $S \subset \mink$. In other words, every
timelike minimal surface in $\mink$ is the lift of a Chebyshev net of
first kind:

\begin{theo}
  For each timelike minimal surface $S \subset \mink$ and each point
  $P_{0} \in S$ there exists an open connected subset
  $I \times J \subset \real^{2}$ and a function
  $f \colon  I \times J \longrightarrow \mink$ such that $f(I \times J)$ is an
  open subset of the surface $S$, where
  \begin{equation}\label{eq:local_min_timelike}
    f(u,v) = P_{0} + (u + v) \partial_{0} + \int_{0}^{u} n_{0}(\xi)\,
    \dd\xi + \int_{0}^{v} n_{3}(\xi)\,\dd\xi,
  \end{equation} 
  and $n_{0} \colon  I \longrightarrow S^{2}$ and
  $n_{3} \colon  J \longrightarrow S^{2}$ are smooth curves on the unit sphere
  of the Euclidean space $\mathbb{E}$ such that
  $\{(u,v) \in I \times J \colon  \vert \lpr{n_{0}(u)}{n_{3}(v)} \vert = 1\} =
  \emptyset.$
\end{theo}

\begin{proof}
  It is well known (see~\cite[p. 68]{C}) that any open neighborhood of a
  timelike surface of $\mink$ admits a parametrization given by a sum
  of two lightlike curves
  \[p(t,s) = P_{0} + X(t) + Y(s),\] where
  $X(t) = X_{0}(t) \partial_{0} + \hat{X}(t)$ and
  $Y(s) = Y_{0}(s) \partial_{0} + \hat{Y}(s)$, for curves 
  $\hat{X}(t),\hat{Y}(s)\in\mathbb{E}$, and
  \[\frac{d}{dt}X_{0}(t) > 0 \quad\mbox{and}\quad\frac{d}{ds}Y_{0}(s) > 0,\] 
  for each $(t,s) \in I' \times J'$. We define the functions $t = t(u)$
  and $s = s(v)$ for $(u,v) \in I \times J$ such that
  \[f(u,v) = P_{0} + (u + v) \partial_{0} + \hat{X}(t(u)) +
    \hat{Y}(s(v)),\] $n_{0}(u) = \frac{d}{du}(\hat{X}(t(u)))$ and
  $n_{3}(v) = \frac{d}{dv}(\hat{Y}(s(v)))$.
\end{proof}

\begin{corol}
  If $(I \times J,f)$ is given by formula (\ref{eq:local_min_timelike})
  and $w = (u,v) \in I \times J$ then,
  \[\frac{\partial f}{\partial u}(u,v) = \partial_{0} + n_{0}(u) =
    l_{0}(u)\quad\mbox{and}\quad\frac{\partial f}{\partial v}(u,v) =
    \partial_{0} + n_{3}(v) = l_{3}(v)\] are lightlike vectors, the
  induced metric is
  $\dd s_{f}^{2} = (-1 + \cos \theta(w))\, \dd u\dd v$, and the normal
  bundle has a basis given by Theorem \ref{teo:tcheblift_geom} and
  formulas (\ref{eq:convenient_base}):
  \begin{align*}
    \tilde{e}_{1}(w)&= \cot\big( \theta(w)/2\big) \; \partial_{0} +
    \cosec
    \big(\theta(w)/2\big) \; e(w)\quad\mbox{and}\\
    e_{2}(w) &= \frac{1}{\sin \theta(w)} n_{0}(u) \times_{\real^{3}}
    n_{3}(v),
  \end{align*}
  where
  $e(w) = \dfrac{1}{2 \cos\big(\theta(w)/2\big)}\big(n_{0}(u) +
  n_{3}(v)\big) \in S^{2}.$ The immersion $(I \times J,X)$ defined by
\begin{equation}
X(w) = \int_{0}^{u} n_{0}(\xi)\,\dd\xi + \int_{0}^{v} n_{3}(\xi)\,\dd\xi,
\end{equation} 
is then a Chebyshev net of first kind.
\end{corol}

Now we can establish our main result:

\begin{theo}\label{teo:main}
  Let $c\colon I\subset\real\to\mink$,
  $c(t) = (c_{0}(t),c_{1}(t),c_{2}(t),c_{3}(t))$ be a given real
  analytic lightlike curve , and $\D(t) = \Span\big\{a(t),b(t)\big\}$ a
  normal and orthonormal spacelike distribution along this curve. A
  necessary and sufficient condition for the existence of a timelike
  minimal immersion $(I \times J,f)$ such that $f(t,0) = c(t)$ and the
  normal space along $c(t)$ is $N_{c(t)}f(M) = \D(t)$ is
  \begin{equation}\label{eq:bjorling_cond}
    c'(t) = c'_{0}(t)\big(\partial_{0} + n_{0}(t)\big)
  \end{equation} 
  where $n_{0}(t) = \pi(\tau(t) - \nu(t))$,
  $\pi$ is the projection defined by (\ref{eq:projection}), and the
  vectors $\tau$ and $\nu$ are given by (\ref{eq:tau}) and (\ref{eq:nu}),
  respectively.
\end{theo}

\begin{proof}
  The condition is necessary: if we have such an immersion, it can be
  written as $f(t,s) = P_0+X(t) + Y(s)$ and, from $f(t,0) = c(t)$ it
  follows that $c'(t) = f_{t}(t,s) = X_{t}(t)$ for each $s \in J$, with
  $\lpr{X_{t}(t)}{X_{t}(t)} = 0$.  The normal bundle of $f(I \times J)$,
  $\D(t,s)$, restricted to the curve, that is $s=0$, implies that $c'(t)$
  defines a lightlike direction orthogonal to $\D(t,0)$.  Let $l_{0}(t)$
  be this direction. Then $c'(t)$ and $l_{0}(t)=\partial_0+n_0(t)$ must
  be parallel to each other. The scalar in (\ref{eq:bjorling_cond}) is
  $c'_0(t)$, since the first coordinate of $l_0(t)$ is 1.

  The condition is also sufficient. Up to a changing of variables
  $t\leftrightarrow u$, if needed, we can suppose that
  $c'(u) = l_{0}(u)$.
  This defines a lightlike vector field $l_3$ along the curve, whose
  first coordinate is 1 and such that $\lpr{l_0(u)}{l_3(u)}<0$ and the
  vector field $n_3(u)=l_3(u)-\partial_0=\pi(\tau+\nu)\in S^2$.

  Now we need to extend the distribution $\D$, defined on $I$ to
  $\D(u,v)$, defined on $I\times J$.% , such that the vector field $l_3$ is

  To do so, consider the curve
  \begin{equation}
    \alpha(u) = c(u) - u
    \partial_{0}\in\{k\}\times\real^3\equiv\mathbb{E},\mbox{ for some 
      $k\in\real$},
  \end{equation} 
  and let $\mathcal{F} = \{T(u) = n_{0}(u), N(u), B(u)\}$ be its Frenet
  frame. Since $\mathcal{F}$ is a basis of $\mathbb{E}$, there are
  functions $p,q\colon I\to\real$ such that, along $\alpha$, we have
  \begin{equation}
    \label{eq:n3(u)}
    n_3(u)=\cos \theta(u) T(u) + p(u) N(u) + q(u) B(u).
  \end{equation}
  In particular, $p^2(u)+q^2(u)=\sin^2\theta(u)$.

  Our aim is to provide extensions of the vector fields $n_0$ and $n_3$
  to $I\times J$ such that $n_0(u,v)=n_0(u)$ and $n_3(u,v)=n_3(v)$. For
  this, if such extension exists for $n_3$, we can extend, using the
  same notation, all of the functions in the coefficients of
  (\ref{eq:n3(u)}) to $I\times J$. The Frenet formulae for $\alpha$ lead
  to
  \[0 = - (\theta_u \sin \theta) T + (\kappa \cos \theta) N + p_{u} N +
    p(-\kappa T + \tau B) + q_{u} B - q \tau N,\] where $\kappa(u)$ and
  $\tau(u)$ are, respectively, the curvature and the torsion of
  $\alpha$.
  Hence the desired extensions must satisfy the following PDE system:
  \begin{equation}
    \label{eq:suffsystem}
    \begin{cases}
      \theta_u(u,v) \sin \theta(u,v) + \kappa(u) p(u,v) &= 0 \\
      p_u(u,v)+\kappa(u) \cos \theta(u,v) - \tau(u) q(u,v)&= 0 \\
      q_u(u,v) + \tau(u) p(u,v) &= 0,
    \end{cases}
  \end{equation}
  with initial conditions $p(u,0)=p(u), q(u,0)=q(u)$ and
  $\theta(u,0)=\theta(u)$ along the interval $I$. Since
  $p^2(u,v)+q^2(u,v)=\sin^2\theta(u,v)$ the above system is equivalent
  to
  \begin{equation}
    \label{eq:suffsystem2}
    \begin{cases}
      \kappa(u) p(u,v) &= -\theta_u(u,v) \sin \theta(u,v) \\
      \tau(u) q(u,v)&= p_u(u,v)+\kappa(u) \cos \theta(u,v) \\
      p^2(u,v)+q^2(u,v)&=\sin^2\theta(u,v),
    \end{cases} 
  \end{equation}
  with the same initial conditions. Hence, for each extension of the
  function $\theta$ to $I\times J$ we have functions $p,q$ determined.

  We set
  \[n_{3}(v) = \cos \theta(u,v) T(u) + p(u,v) N(u) + q(u,v) B(u),\]
  which depends, by construction, only on $v$ allowing us to build the
  tangent lightlike vector, $l_{3}(v)$.
  In this way the immersion $f\colon I\times J\to\mink$ given by
  (\ref{eq:local_min_timelike}) is a local solution to Question
  \ref{problem:bjorling}.
\end{proof}

In system (\ref{eq:suffsystem2}) if $\theta_u(u,v)\neq 0$ we see that
$\theta_u(u,v)=-\kappa(u)$ or $p(u,v)\equiv0$, and $q(u,v)\equiv
0$.
Since $p$ and $q$ cannot both vanish simultaneously, we have from last
equation in (\ref{eq:suffsystem}) that $\tau(u)\equiv 0$, that is
$\alpha$ is a planar curve.

On the other side, if $\theta_u(u,v)\equiv0$ then either
$\kappa(u)\equiv0$ or $p(u,v)\equiv0$. The former case says the $\alpha$
is a straight line in $\mathbb{E}$, implying that $c(u)$ is a lightlike
straight line in $\mink$. Here the immersion has the form
\[f(u,v)=u\vec{l_0}+v\partial_0+\int_0^v n_3(\xi)\,\dd\xi.\] for some
constant lightlike vector $\vec{l_0}$. In the latter case,
$q(u,v)=\sin\theta(u,v)$ and, noting that $\theta(u,v)=\theta(v)$, we
have $\tan(\theta(v))=\kappa(u)/\tau(u)$. That is, both $\theta(u,v)$
and $\kappa(u)/\tau(u)$ are constants. In particular $\alpha$ is an
helix.  From equation (\ref{eq:K_lift}) in Proposition~\ref{prop:K_lift}
we have that such surfaces are planar. From (\ref{eq:K_liftK_T}) we
conclude that this timelike surface is the lift of a planar Chebyshev
net in $\mathbb{E}$.

We finally observe that we obtain existence and non-uniqueness of
solutions for the Björling problem in $\lorentz^3=\real^3_1$ with
initial data given by the lightlike curve $\gamma\colon I\to\lorentz^3$ and
normal vector field $n\colon I\to S^2$, using Theorem \ref{teo:main} with
$c(t)=(\gamma(t),0)$, $a(t)=(n(t),0)$ and $b(t)=e_4$. An explicit
example of non-uniqueness is Example 3.2 in~\cite{CDM}.

\bibliographystyle{plain}
\bibliography{padua_lymber-llbj-bib}

\end{document}